\documentclass[11pt,english]{smfart}
   
\usepackage{eucal,enumerate,dsfont,txfonts}

\usepackage[text={6.5in,9in},centering]{geometry} 
\usepackage{dsfont,eucal,mathptmx,txfonts}

\usepackage[dvipsnames]{xcolor}   
\usepackage{xparse}
\usepackage{xr-hyper}
\definecolor{brightmaroon}{rgb}{0.76, 0.13, 0.28}
\usepackage[linktocpage=true,colorlinks=true,hyperindex,citecolor=teal,linkcolor=blue]{hyperref} 

\renewcommand{\theequation}{\thesection.\arabic{equation}}

\newtheorem{thm}{Theorem}[section]
\newtheorem{exam}[thm]{Example}

\newtheorem{prop}[thm]{Proposition}
 
\newtheorem{lem}[thm]{Lemma}
\newtheorem{cor}[thm]{Corollary}

\newtheorem{rem}[thm]{Remark}

\numberwithin{equation}{section}

\usepackage{verbatim}   

\newcommand{\Max}{\mathrm{Max}} 

\author{Carmelo A. Finocchiaro}  
\address{Dipartimento di Matematica e Informatica, Universit\`{a} degli Studi di Catania, 95125 Catania, Italy
} 
\email{cafinocchiaro@unict.it} 

\author{Amartya Goswami}
\address{[1] Department of Mathematics and Applied Mathematics\\
University of Johannesburg, Auckland Park Kingsway Campus,
P.O. Box 524, 
Auckland Park 
2006, 
South Africa. [2] National Institute for Theoretical and Computational Sciences (NITheCS), South Africa.}

\email{agoswami@uj.ac.za}  

\author{Dario Spirito}
\address{Dipartimento di Scienze Mathematiche, Informatiche e Fisiche, University of Udine, via delle Scienze 206, 33100 Udine, Italy}

\email{dario.spirito@uniud.it}
  
\title{Distinguished classes of ideal spaces and their topological properties} 

\subjclass{13C05, 54B35, 13A15} 

\keywords{coarse lower topology, finitely generated ideal, nilpotent ideal, regular ideal, primary ideal, quasi-compactness, sobriety, spectral space} 
   
\begin{document}
 
\begin{abstract}
We consider the set of all the ideals of a ring, endowed with the coarse lower topology. The aim of this paper is to study topological properties of distinguished subspaces of this space and detect the spectrality of some of them. 
\end{abstract} 

\maketitle
  
\section{Introduction}

Let $(X,\leqslant)$ be a partially ordered set.  The coarse lower topology on $X$ is the topology whose subbasic closed sets are the sets of the type
$$
\{x\}^\uparrow:=\{y\in X\mid x\leqslant y\},
$$
for $x$ varying in $X$. On the other hand, any T$_0$ topology $\mathcal T$ on a space $S$ determines a partial order $\preceq_{\mathcal T}$, called \emph{the specialization order induced by $\mathcal T$}, defined by setting, for every, $s,t\in S$,
$$
s\preceq_{\mathcal T} t:\iff t\in \overline{\{s\}}. 
$$
Thus the coarse lower topology on a partially ordered set $(X,\leqslant)$ is the coarsest topology on $X$ whose specialization order is $\leqslant$. Topologies on posets have intensively been studied, see  \cite{DFP81, D82, N01, XZ17, Y10, LL20, MJH17, W07} for a deeper insight on this circle of ideas. A natural setting where to apply this general framework is that of the set ${\rm Idl}(R)$ of all the ideals of a ring $R$, partially ordered by inclusion. It is well known that  ${\rm Idl}(R)$, endowed with the coarse lower topology, is a \emph{spectral space}, that is, it is homeomorphic to the prime spectrum of a ring; note that the subspace topology on ${\rm Spec}(R)$ induced by the coarse lower topology is the classical Zariski topology of the prime spectrum of a ring. In some recent investigation some other  examples of spectral subspaces of ${\rm Idl}(R)$ have been detected (see \cite{FFS2016, FS2020}), by observing that they are closed sets in the constructible topology of ${\rm Idl}(R)$; the aim of this paper is to find other spectral subspaces of ${\rm Idl}(R)$ without using the constructible topology. This leads to check explicitly the topological properties of Hochster's criterion for spectrality on distinguished classes of ideals. Among the other things, a classification of all Noetherian rings for which the space of primary ideals is spectral is provided.

\section{Topological preliminaries}
Let $X$ be a topological space, and let $Y\subseteq X$ be a nonempty subset. Then, $Y$ is \emph{irreducible} if, whenever $Y\subseteq V_1\cup V_2$ for some closed subsets $V_1,V_2$ of $X$, we have $Y\subseteq V_1$ or $Y\subseteq V_2$. A \emph{generic point} of $Y$ is an element $y\in X$ such that $Y$ is equal to the closure $\overline{\{y\}}$ of the point $y$. If every irreducible closed subset of $X$ has a generic point, $X$ is said to be \emph{sober}.

A \emph{spectral space} is a topological space that is homeomorphic to the prime spectrum of some commutative ring, endowed with the Zariski topology. As it is proved in \cite{H69}, spectral spaces can be characterized purely in topological terms. Precisely, a topological space $X$ is spectral if and only if it is quasi-compact, sober and it has a basis of quasi-compact open sets that is closed under finite intersections. 

Given a spectral space $X$, it is possible to define from the starting topology two new topologies, the \emph{inverse} and the \emph{constructible} topology; in the former, the specialization order is the opposite of the one of the starting topology, while the latter is Hausdorff (and thus its specialization order is trivial). A subset that is closed with respect to the inverse or the constructible topology is spectral also when seen in the starting topology; in particular, subsets that are closed in the constructible topology (called \emph{proconstructible} subsets) provide many examples of spectral spaces. See \cite{H69} and \cite[Chapter 1]{DST19} for the precise definition and for properties of these two topologies.

We end this section with a useful lemma.
\begin{lem}\label{upper-bound}
Let $X$ be a quasi-compact T$_0$ space. Then every chain in $X$ has an upper bound. 
\end{lem} 

\begin{proof}
Let $C\subseteq X$ be a chain and let $\mathcal G:=\left\{\overline{\{c\}}\mid c\in C \right\}$. Then $\mathcal G$ is clearly a chain of nonempty closed sets of $X$ and thus it has the finite intersection property. Since $X$ is quasi-compact, there is a point $z\in \cap \mathcal G$ and, by definition, $c\leqslant z$, for every $c\in C$. The conclusion follows. 
\end{proof}

\section{The coarse lower topology} 
All rings considered in this paper are assumed to be commutative and to possess an identity element. We denote the radical of an ideal $\mathfrak{a}$ by $\sqrt{\mathfrak{a}}.$ An ideal $\mathfrak{a}$ is called \emph{proper} if $\mathfrak{a}$ is not equal to $R$. 

Let $\mathrm{Idl}(R)$ denote the set of all the ideals of $R$. The coarse lower topology on $\mathrm{Idl}(R)$ will be the topology for which the sets of the type 
$$
\{\mathfrak a \}^{\uparrow}:=\{\mathfrak i\in \mathrm{Idl}(R)\mid \mathfrak a\subseteq \mathfrak i \}
$$
(where $\mathfrak a$ runs among the ideals of $R$) form a subbasis of closed sets.

We call a subset of $\mathrm{Idl}(R)$, endowed with the subspace topology induced by the coarse lower topology, an \emph{ideal space} of $R$. Some examples of ideal spaces defined in an algebraic way are the following: prime ideals ($\mathrm{Spec}(R)$),  
maximal ideals ($\mathrm{Max}(R)$), proper ideals 
($\mathrm{Prp}(R)$),  radical ideals 
($\mathrm{Rad}(R)$),  minimal ideals
($\mathrm{Min}(R)$),  
minimal prime ideals ($\mathrm{Spn}(R)$), primary ideals
($\mathrm{Prm}(R)$), 
nil ideals
($\mathrm{Nil}(R)$),  nilpotent ideals
($\mathrm{Nip}(R)$),  irreducible ideals
($\mathrm{Irr}(R)$),   completely irreducible ideals
($\mathrm{Irc}(R)$) (in the sense of \cite{FHO06}), principal ideals
($\mathrm{Prn}(R)$),  proper regular ideals
($\mathrm{Reg}(R)$),  proper finitely generated ideals
($\mathrm{Fgn}(R)$), strongly  irreducible ideals ($\mathrm{Irs}(R)$). We reserve the symbol $\mathrm{X}(R)$ to denote an arbitrary ideal space of a ring $R$.

In view of \cite[7.2.12]{DST19}, $\mathrm{Idl}(R)$ is spectral, with the coarse lower topology, since it is a complete algebraic lattice and its inverse topology coincides with the Zariski topology considered in \cite[Section 5]{FS2020}. Notice that the subset $\mathrm{Prp}(R)$ of $\mathrm{Idl}(R)$ consisting of all proper ideals of $R$ is proconstructible in $\mathrm{Idl}(R)$ (see the paragraph of \cite{FS2020} after Proposition 5.1) and thus, in particular, it is spectral as a subspace of $\mathrm{Idl}(R)$. For an alternative proof of spectrality of $\mathrm{Prp}(R)$, see \cite{G22}.

The following is a ``radical-like'' criterion to establish if an ideal space is sober, that will be useful later.
 \begin{thm}\label{sober}\cite[Theorem 3.18]{DG22} 
Let $\mathrm{ X}$ be a subspace of $\mathrm{ Idl}(R)$. For every ideal $\mathfrak a$ of $R$, let $$\sqrt[\mathrm{ X}]{\mathfrak a}:=\bigcap\left\{\mathfrak b\mid \mathfrak b\in \mathrm X\cap \{\mathfrak a\}^\uparrow \right\}.$$ Then $\mathrm X$ is a sober space if and only if whenever $\mathfrak a$ is an ideal of $R$ and $\mathrm X\cap \{\mathfrak a\}^\uparrow$ is irreducible, then $\sqrt[\mathrm{ X}]{\mathfrak a}\in \mathrm X$. 
 \end{thm} 
 The following fact characterizes quasi-compact ideal spaces. 
 
 \begin{prop}
Let $R$ be a ring, let $\mathrm X\subseteq \mathrm{Idl}(R)$ and let $\Max(\mathrm{X})$ denote the set of maximal elements of $\mathrm{X}$. Then the following conditions are equivalent: 
\begin{enumerate}[\upshape (i)]
\item $\mathrm{X}$ is quasi-compact;

\item for every $\mathfrak x\in\mathrm{X}$ there is $\mathfrak y\in\Max(\mathrm{X})$ such that $\mathfrak x\subseteq \mathfrak  y$, and $\Max(\mathrm{X})$ is quasi-compact.
\end{enumerate}
 \end{prop}
 
 \begin{proof}
 	(ii)$\Rightarrow$(i). Let $\mathcal{U}$ be an open cover of $\mathrm{X}$. Then, $\mathcal{U}$ also covers $\Max(\mathrm{X})$, and thus there is a finite subcover $\mathcal{U}'$ of $\Max(\mathrm{X})$. Take $\mathfrak x\in\mathrm{X}$: by hypothesis, there is $\mathfrak y\in\Max(\mathrm{X})$ such that $\mathfrak x\subseteq \mathfrak y$, and $U\in\mathcal{U}'$ such that $\mathfrak y\in U$. Then, $\mathfrak x\in U$, and thus $\mathcal{U}'$ is a finite subcover also for $\mathrm{X}$. Hence $\mathrm{X}$ is quasi-compact.
 	
 	(i)$\Rightarrow$(ii). Let $\mathfrak x\in\mathrm{X}$, and consider $\mathrm{X}':=\{\mathfrak{x}\}^\uparrow\cap {\rm X}$. Then, $\mathrm{X}'$ is a closed subset of $\mathrm{X}$, and thus it is quasi-compact; by Lemma \ref{upper-bound}, every ascending chain in $\mathrm{X}'$ is bounded, and hence by Zorn's Lemma $\mathrm{X}'$ has maximal elements, i.e., $\mathfrak{x}$ is contained in some $\mathfrak{y}\in\Max(\mathrm{X})$.
%

 	Let now $\mathcal{U}$ be an open cover of $\Max(\mathrm{X})$. Then, $\mathcal{U}$ is also a cover of $\mathrm{X}$, and thus it admits a finite subcover, which will be also a finite subcover of $\Max(\mathrm{X})$. Thus, $\Max(\mathrm{X})$ is quasi-compact.
 \end{proof}

As a particular case of the previous criterion we get the following known fact. 
 
\begin{cor}{\cite[Theorem 3.11]{DG22}}\label{comp}
Let $R$ be a ring and let $\mathrm X\subseteq \mathrm{Idl}(R)$ such that $\Max(R)\subseteq \mathrm X$. Then $\mathrm X$ is quasi-compact. 
\end{cor}

\begin{cor}  
The ideal spaces $\mathrm{Max}(R),$ $\mathrm{Spec}(R),$ $\mathrm{Irs}(R),$ $\mathrm{Prm}(R),$ $\mathrm{Irr}(R),$ $\mathrm{Irc}(R),$ $\mathrm{Rad}(R),$ $\mathrm{Prp}(R)$ are quasi-compact.
\end{cor}

In case $R$ is a Noetherian ring, the situation regarding quasi-compactness of ideal spaces is much  simpler. 

\begin{prop}\label{noetherian}
For a ring $R$, the following conditions are equivalent. 
\begin{enumerate}[\upshape (i)]
\item $R$ is a Noetherian ring. 

\item $\mathrm{Idl}(R)$ is a Noetherian space. 
\end{enumerate}
\end{prop}

\begin{proof}
(i)$\Rightarrow$(ii). Suppose $V$ is closed and irreducible: then, $\mathfrak{a}:=\bigcap\{\mathfrak{b}\mid \mathfrak{b}\in V\}$ is finitely generated (since $R$ is Noetherian) and thus $V$ must be equal to $\{\mathfrak{a}\}^\uparrow$. In particular, $V$ has a generic point. Thus $\mathrm{Idl}(R)$ is sober.

We show that every subset $X$ of $\mathrm{Idl}(R)$ is quasi-compact. Take a collection $\mathcal G:=\{X\cap \{\mathfrak a_i\}^\uparrow\mid i\in I \}$ of subbasic closed sets of $X$ with the finite intersection property. By assumption, the ideal $\mathfrak b:=\sum_{i\in I}\mathfrak a_i$ is finitely generated, say $\mathfrak b=(\alpha_1,\ldots,\alpha_n)$. For every $1\leqslant j\leqslant n$, there exists a finite subset $H_j$ of $I$ such that $\alpha_j\in \sum_{i\in H_j}\mathfrak a_i$. Thus, if $H:=\bigcup_{j=1}^nH_j$, it immediately follows that $\mathfrak b=\sum_{i\in H}\mathfrak a_i$. Hence we have
$$
\cap\, \mathcal G=X\cap \{\mathfrak b\}^\uparrow=X\cap \left\{\sum_{i\in H}\mathfrak a_i \right\}^\uparrow= \bigcap_{i\in H}X\cap \{\mathfrak a_i \}^\uparrow\neq \emptyset,
$$
since $H$ is finite and $\mathcal G$ has the finite intersection property. Then the conclusion follows by Alexander's subbasis Theorem. 

(ii)$\Rightarrow$(i). Assume that $\mathrm{Idl}(R)$ is Noetherian and that there exists an ideal $\mathfrak a$ of $R$ that is not finitely generated. Then the subspace $X:=\{\mathfrak b\in \mathrm{Fgn}(R)\mid \mathfrak b\subset \mathfrak a \}$ is not quasi-compact. 
As a matter of fact, the collection of closed sets $\mathcal G:=\{X\cap\{aR\}^\uparrow\mid a\in \mathfrak a \}$ of $X$ clearly has the finite intersection property, but has empty intersection. 
\end{proof} 

The following corollary is now immediate. 
\begin{cor}\label{cor-noet}
Let $R$ be a Noetherian ring and let $\mathrm X\subseteq \mathrm{Idl}(R)$. Then $\mathrm X$ is spectral if and only if it is sober. 
\end{cor}

The next proposition is a general result about the relationship between lower directed subsets  and soberness. We recall that, if $X$ is a spectral space and $Z\subseteq X$ is lower directed (in the order induced by the topology), then $Z$ always admits an infimum in $X$ \cite[Proposition 4.2.1]{DST19}.
\begin{prop}\label{inf-sober-1}
Let $X$ be a spectral space, and let $Z\subseteq Y\subseteq X$ be subsets such that $Z$ is lower directed (in the order induced by the topology). If $Y$ is sober, then $\inf Z\in Y$.
\end{prop}
\begin{proof}
Let $z$ be the infimum of $X$ in $Z$ (which exists by \cite[Proposition 4.2.1]{DST19}). Consider $Y':=\{z\}^\uparrow\cap Y$: then, $Y'$ is a closed subset of $Y$. Suppose that $Y'$ is not irreducible: then, there are closed subsets $V_1,V_2$ of $X$ such that $(V_1\cup V_2)\cap Y=Y'$ and such that $Y'$ is not contained in either $V_1$ or $V_2$. If $Z\subseteq V_1$, then $z=\inf Z\in V_1$ and $Y'\subseteq V_1$, a contradiction (and analogously for $V_2$); thus, there are $v_1\in (V_1\cap Z)\setminus V_2$ and $v_2\in (V_2\cap Z)\setminus V_1$. Since $Z$ is lower directed, there is $y\in Z$ such that $y\leqslant v_1$ and $y\leqslant v_2$; by construction, $y$ belongs to one of the $V_i$, say $V_1$. Since $V_1$ is closed, $y\leqslant v_2$ implies that $v_2\in V_1$, a contradiction.

Therefore, $Y'$ is irreducible. Since $Y$ is sober, $Y'$ has a generic point $z'$; moreover, $z'\leqslant z''$ for every $z''\in Z$, and thus $z'\leqslant z$. Since $y\geqslant z$ for every $y\in Y'$, we also have $z'\geqslant z$. Hence $z'=z\in Z$, as claimed.
\end{proof}

\begin{rem}
Proposition \ref{inf-sober-1} applies, in particular, when $Z$ is a chain.
\end{rem}

\section{Some classes of ideal spaces} 
We now discuss some relevant topological properties of some classes of ideal spaces.  

\subsection{Strongly irreducible ideals and their subclasses}  

It follows from Corollary \ref{comp} that the space of strongly irreducible ideals $\mathrm{Irs}(R)$ is quasi-compact and so are $\mathrm{Max}(R)$ and $\mathrm{Spec}(R)$. Since $\mathrm{Spec}(R)$ is spectral, it is also sober. It has been proved in \cite[Proposition 3.19]{DG22} that $\mathrm{Irs}(R)$ is sober.

\subsection{Finitely generated ideals.} 

Given a ring $R$, let $\mathrm{Fgn}(R)$ be the space of proper finitely generated ideals of $R$, endowed with the subspace topology induced by the coarse lower topology of the space $\mathrm{Idl}(R)$ of all the ideals of $R$. 

\begin{prop}
For a ring $R$ the following conditions are equivalent. 
\begin{enumerate}[\upshape (i)]
\item $\mathrm{Fgn}(R)$ is quasi-compact. 
 
\item $\mathrm{Max}(R)\subseteq \mathrm{Fgn}(R)$. 
\end{enumerate}
\end{prop}

\begin{proof}
It is clear that (ii) implies (i), in view of Corollary \ref{comp}.
Conversely, assume that there exists a maximal ideal $\mathfrak m$ of $R$ that is not finitely generated, and consider the collection of closed subspaces $$\mathcal G:=\left\{\{aR\}^\uparrow\cap \mathrm{Fgn}(R)\mid a\in \mathfrak m \right\}.$$  Clearly $\mathcal G$ has the finite intersection property, but $\cap \mathcal G=\emptyset$: indeed, if there exists an ideal $\mathfrak b\in \cap \mathcal G$, then $\mathfrak b$ is finitely generated and contains $\mathfrak m$. Since $\mathfrak m$ is not finitely generated, $\mathfrak b\supsetneq \mathfrak m$ and thus $\mathfrak b=R$, against the fact that $\mathrm{Fgn}(R)$ consists of proper ideals.  
\end{proof}

\begin{exam}
\emph{We now observe that $\mathrm{Fgn}(R)$ can fail to be sober. Let $p$ be any prime number and let $R:=\mathds Z_{(p)}+T\mathds Q[T]_{(T)}$, where $T$ is an indeterminate over $\mathds Q$. Then $R$ is a two-dimensional valuation domain and $$\mathrm{Spec}(R)=\{(0), \mathfrak p:=T\mathds Q[T]_{(T)}, \mathfrak m:=pR\}.$$ It is well known that $\mathfrak p$ is not a finitely generated ideal of $R$ (see e.g. \cite[Theorem 17.3(a)]{GI72}). Now consider the nonempty closed subset $C:=\{\mathfrak p\}^\uparrow\cap \mathrm{Fgn}(R)$ and notice that $C$ is irreducible. Indeed, if $$U:=\mathrm{Fgn}(R)\setminus \{\mathfrak a\}^\uparrow, U'=\mathrm{Fgn}(R)\setminus \{\mathfrak b\}^\uparrow$$ are subbasic open sets of $\mathrm{Fgn}(R)$ ($\mathfrak a,\mathfrak b$ are ideals of $R$) and $U\cap C,U'\cap C\neq \emptyset$, then $U\cap U'\cap C\neq \emptyset$, because $U,U'$ are comparable since all ideals of $R$ are comparable. Since every nonzero non-maximal prime ideal of a valuation domain is divisorial (see e.g. \cite[Corollary 4.1.13]{fontana_libro}), $\mathfrak{p}$ is the intersection of all principal ideals containing it. In particular, $\sqrt[\mathrm{Fgn}(R)]{\mathfrak p}=\mathfrak p\notin \mathrm{Fgn}(R)$. By Theorem \ref{sober}, $\mathrm{Fgn}(R)$ is not sober. }
\end{exam}

\subsection{Nilpotent ideals.} 
Recall that an ideal $\mathfrak a$ of a ring $R$ is \emph{nilpotent} if $\mathfrak a^k=0$ for some positive integer $k$. The following example will show that the space $\mathrm{Nip}(R)$ of nilpotent ideals of $R$ can easily fail to be quasi-compact.

\begin{exam}
\emph{Let us consider the ring $R:=\prod_{i\geqslant 2}\mathds Z/2^i\mathds Z$ and set $$f_1:=(\overline{2},0,0,0,\ldots), \; f_2:=(\overline{2},\overline{2},0,0,\ldots), \; f_3:=(\overline{2},\overline{2},\overline{2},0,\ldots),$$ and so on. For every positive integer $i$, consider the nilpotent ideal $\mathfrak a_i:=f_iR$ (notice that $\mathfrak a_i^{i+1}=0$ and that $\mathfrak a_i^i\neq 0$). Thus we get the ascending chain $\mathfrak a_1\subset \mathfrak a_2\subset \mathfrak a_3\subset \cdots$ in the space $\mathrm{Nip}(R)$ and such a chain has no upper bounds: indeed, if $\mathfrak b$ is any nilpotent ideal and $k$ is the minimum positive integer $n$ such that $\mathfrak b^n=0$ then $\mathfrak a_k\nsubseteq \mathfrak b$, since $f_k^k\neq 0$. Then the conclusion immediately follows from Lemma \ref{upper-bound}}.  
\end{exam}

\subsection{Regular ideals.}
Recall that an ideal of a ring $R$ is \emph{regular} if it contains a regular element, i.e., an element that is not a zero-divisor of $R$. Let $\mathrm{Reg}(R)$ denote the subspace of $\mathrm{Idl}(R)$ consisting of all proper regular ideals. First notice that $\mathrm{Reg}(R)$ is closed under specialization in the spectral space $\mathrm{Prp}(R)$. 

\begin{prop}
Let $R$ be a ring satisfying at least one of the following conditions. 
\begin{enumerate} [\upshape (i)] 
\item $R$ is Noetherian. 

\item $R$ is local.

\item Every maximal ideal of $R$ is regular. 
\end{enumerate}
Then $\mathrm{Reg}(R)$ is quasi-compact. 
\end{prop}

\begin{proof}
Cases (i) and (iii) immediately follows from Proposition \ref{noetherian} and Corollary  \ref{comp}, respectively. Now suppose $R$ is local with maximal ideal $\mathfrak m$. If $\mathfrak m$ is regular, the conclusion follows again by Corollary \ref{comp}. In case $\mathfrak m$ consists of zero-divisors, every regular element is invertible, and thus $\mathrm{Reg}(R)=\emptyset$ is quasi-compact.  
\end{proof}

We now give two example showing that $\mathrm{Reg}(R)$ can fail to be quasi-compact.
\begin{exam}
\emph{Let $D$ be a one-dimensional domain such that $\mathrm{Spec}(D)$ is non-Noetherian, and let $\mathfrak m_\infty$ be a maximal ideal of $D$ that is not the radical of any finitely generated ideal of $D$; for example, $D$ may be any almost Dedekind domain that is not Noetherian (see e.g. \cite{loper-almostdedekind} for several constructions of this kind of rings). Let $K:=D/\mathfrak m_\infty$, consider the $D$-module  $X:=K^{(\mathds N)}$ and let $R:=D\times X$ endowed with the following multiplication: 
$$
(a,k)(b,l):=(ab,al+bk+kl),
$$
for every $(a,k),(b,l)\in R$. Then $\widetilde{\mathfrak m_\infty}:=\mathfrak m_\infty\times X$ is a maximal ideal of $R$ consisting of zero-divisors, by \cite[Theorem 8.3(f)]{LU2005}. Now consider elements $(a_1,\varphi_1),\ldots (a_n,\varphi_n)\in \widetilde{\mathfrak m_\infty}$. By assumption, there exist a maximal ideal $\mathfrak n\neq \mathfrak{m}_\infty$ such that $a_1,\ldots,a_n\in \mathfrak n$. In particular, the elements $(a_1,\varphi_1),\ldots (a_n,\varphi_n)$ belong to the maximal ideal $\widetilde{\mathfrak n}:=\mathfrak n\times X$ of $R$ and the fact that $\mathfrak m\neq \mathfrak n$  implies that $\widetilde{\mathfrak n}$ is regular, again by \cite[Theorem 8.3(f)]{LU2005}. 
It immediately follows that the space ${\rm Reg}(R)$ of regular proper ideals of $R$ is not quasi-compact. Indeed, the collection of closed sets $$\mathcal G:=\{{\rm Reg}(R)\cap \{fR\}^\uparrow\mid f\in \widetilde{\mathfrak m_\infty} \}$$ has the finite intersection property and empty intersection. }
\end{exam}

\begin{exam}
\emph{Let $D$ be a one-dimensional domain such that $\mathrm{Spec}(D)$ is non-Noetherian, and let $\mathfrak m_\infty$ be a maximal ideal of $D$ that is not the radical of any finitely generated ideal of $D$. Let $R:=D[X]/(X\mathfrak{m}_\infty)$, and let $\pi:D[X]\to R$ be the quotient map.}
 
\emph{Consider the collection $$\mathcal G:=\left\{\mathrm{Reg}(R)\cap \{\pi(f)R\}^\uparrow\mid f\in \mathfrak{m}_\infty[X]\} \right\}\cup\left\{\mathrm{Reg}(R)\cap \{\pi(X)R\}^\uparrow\right\}$$ of closed subsets of $\mathrm{Reg}(R)$. The intersection of all elements of $\mathcal{G}$ is empty: indeed, if $\mathfrak{a}$ contains all $\pi(f)$ and $\pi(X)$, then it must be $\mathfrak{b}:=\pi(X,\mathfrak{m}_\infty)$, which is a maximal ideal containing only zero-divisors. On the other hand, if $\mathcal{G}'$ is a finite subset of $\mathcal{G}$, say $$\mathcal{G}':=\left\{\mathrm{Reg}(R)\cap \{\pi(f_i)R\}^\uparrow\mid i=1,\ldots,n\right\}\cup\left\{\mathrm{Reg}(R)\cap \{\pi(X)R\}^\uparrow\right\},$$ then there is an ideal $\mathfrak{n}$ of $D$ containing $f_1,\ldots,f_n$, and thus $\cap\mathcal{G}'$ contains the ideal $\pi((\mathfrak{n},X))$, which is regular (every $g\in\mathfrak{n}\setminus\mathfrak{m}_\infty$ becomes regular in $R$). Hence, $\mathrm{Reg}(R)$ is not quasi-compact.}
\end{exam}

\begin{exam}
\emph{We now prove that $\mathrm{Reg}(R)$ can fail to be sober. Clearly if $R=\mathds Z$ then $\mathrm{Reg}(R)=\mathrm{Prp}(R)\setminus\{(0)\}$. If $n,m$ are nonzero integers and $p$ is a prime number that does neither divide $n$ nor $m$, then $$p\mathds Z\in \mathrm{Reg}(R)\setminus(\{n\mathds Z\}^\uparrow\cup\{m\mathds Z\}^\uparrow).$$ This proves that $\mathrm{Reg}(R)$ is an irreducible space. Since clearly $$\sqrt[\mathrm{Reg}(R)]{(0)}=(0) \notin \mathrm{Reg}(R),$$ from Theorem \ref{sober} we immediately infer that $\mathrm{Reg}(R)$ is not a sober space.}
\end{exam}

\subsection{Primary ideals}
\begin{lem}\label{primary-dim0-irrid}
Let $R$ be a zero-dimensional ring that is not local. Then $\mathrm{Prm}(R)$ is not irreducible.
\end{lem}

\begin{proof}
Since $R$ is zero-dimensional, the Zariski topology and the constructible topology on ${\rm Spec}(R)$  are the same topology and thus ${\rm Spec}(R)$ is totally disconnected, by \cite[Chapter 3, Exercises 11, 28 and 30]{atiyah}. Hence, since $R$ is not local, ${\rm Spec}(R)$ is not connected and thus, in view of 
 \cite[Chapter 1, Exercise 22]{atiyah},
there are nontrivial rings $R_1,R_2$ such that $R$ is isomorphic to the direct product $R_1\times R_2$. In the latter, every primary ideal contains either $(1,0)$ or $(0,1)$, and thus $$\mathrm{Prm}(R_1\times R_2)=(\{(1,0)\}^\uparrow\cap\mathrm{Prm}(R_1\times R_2))\cup(\{(0,1)\}^\uparrow\cap\mathrm{Prm}(R_1\times R_2)).$$ Hence $\mathrm{Prm}(R)$ is not irreducible.
\end{proof}

\begin{prop}\label{sober-dim1}
Let $R$ be a ring that is either:
\begin{itemize}

\item[$\bullet$] a zero-dimensional ring;
\item[$\bullet$] a one-dimensional integral domain.
\end{itemize}
Then $\mathrm{Prm}(R)$ is sober. 
\end{prop}

\begin{proof}
Let $\mathfrak a$ be a non-primary ideal of $R$. Then, $R':=R/\mathfrak{a}$ is zero-dimensional under both hypothesis (if $R$ is a one-dimensional domain, $(0)$ is primary), and the quotient map $R\to R'$ induces a homeomorphism $\{\mathfrak{a}\}^\uparrow\cap\mathrm{Prm}(R)\to\mathrm{Prm}(R')$; by Lemma \ref{primary-dim0-irrid}, the latter is not irreducible, and thus neither $\{\mathfrak{a}\}^\uparrow\cap\mathrm{Prm}(R)$ is irreducible.
Therefore, if $\{\mathfrak{a}\}^\uparrow\cap\mathrm{Prm}(R)$ is irreducible then $\mathfrak{a}$ is primary; thus, $\sqrt[\mathrm{Prm}(R)]{\mathfrak{a}}=\mathfrak{a}\in\mathrm{Prm}(R)$. By Theorem \ref{sober}, $\mathrm{Prm}(R)$ is sober.
\end{proof}

\begin{lem}\label{prm-localization}
Let $R$ be a ring such that $\mathrm{Prm}(R)$ is sober. Then $\mathrm{Prm}(R_{\mathfrak{ m}})$ is sober, for every maximal ideal $\mathfrak m$ of $R$. 
\end{lem}

\begin{proof}
Given a maximal ideal $\mathfrak{ m}$ of $R$, it is immediate that the localization mapping $R\to R_{\mathfrak m}$ induces a homeomorphism of $\mathrm{Prm}(R_{\mathfrak m})$ and $X:=\{\mathfrak a \in \mathrm{Prm}(R)\mid \mathfrak a \subseteq \mathfrak m \}$. Take an ideal $\mathfrak{ i}$ of $R$ such that $X\cap \{ \mathfrak i\}^\uparrow$ is irreducible. Then the closure $\Gamma$ of $X\cap \{ \mathfrak i\}^\uparrow$ in $\mathrm{Prm}(R)$ is irreducible too and thus, by assumption, there exists a primary ideal $\mathfrak a_0$ of $R$ such that $\Gamma=\{\mathfrak a_0\}^\uparrow\cap \mathrm{Prm}(R)$. The inclusion $X\cap \{\mathfrak i\}^\uparrow\subseteq \Gamma$ immediately implies that $\mathfrak m\supseteq \sqrt[X]{\mathfrak i}\supseteq \mathfrak a_0$ (in particular, $\mathfrak a_0\in X$). Conversely, take an element $\alpha\in \sqrt[X]{\mathfrak i}$. Then $C:=\{\alpha R\}^\uparrow\cap {\rm Prm}(R)$ is a closed set of ${\rm Prm}(R)$ containing $X\cap \{ \mathfrak i\}^\uparrow$ and thus $C$ contains $\Gamma$. In particular, $\alpha\in \mathfrak{a}_0$. This proves that $\sqrt[X]{\mathfrak i}=\mathfrak{a}_0$ and thus the conclusion follows from Theorem \ref{sober}. 
\end{proof}

\begin{prop}\label{prm-sober-local}
Let $R$ be a Noetherian local ring. Then ${\rm Prm}(R)$ is sober if and only if ${\rm Prm}(R)={\rm Prp}(R)$. 
\end{prop}

\begin{proof}
Recall that ${\rm Prp}(R)$ is a sober space, since it is spectral. Conversely, suppose that ${\rm Prm}(R)$ is sober and assume, by contradiction, that there exists a proper non-primary ideal $\mathfrak i$ of $R$. Let $\mathfrak n$ and $\overline{\mathfrak n}$ be the maximal ideals of the local rings $R$ and $R/\mathfrak i$, respectively. Since $R/\mathfrak i$ is Noetherian, we get $\bigcap_{n\geqslant 1}\overline{\mathfrak n}^n=(0)$, that is, $\bigcap_{n\geqslant 1}(\mathfrak n^n+\mathfrak i)=\mathfrak i$. Since each ideal of the type $\mathfrak n^n+ \mathfrak i$ is $\mathfrak n$-primary and ${\rm Prm}(R)$ is sober, $\mathfrak i=\inf\{ \mathfrak n^n+ \mathfrak i\mid n\geqslant 1\}$ is primary, by virtue of Proposition \ref{inf-sober-1}, a contradiction. 
\end{proof}

\begin{cor}\label{dim0-1}
Let $R$ be a Noetherian ring. If ${\rm Prm}(R)$ is sober, then $\dim(R)\leqslant 1$. 
\end{cor}

\begin{proof}
Suppose, by contradiction, that there exists a maximal ideal $\mathfrak{m}$ of $R$ such that $\dim(R_{\mathfrak m})\geqslant 2$. It follows that the local ring $R_{\mathfrak m}$ has proper ideals that are not primary \cite[Theorem 4.4]{generalized-primary} and thus ${\rm Prm}(R_{\mathfrak m})$ is not sober, by Proposition \ref{prm-sober-local}. This is a contradiction, by Lemma \ref{prm-localization}. 
\end{proof}

\begin{cor}\label{prm-domain1}
Let $R$ be a one-dimensional Noetherian local ring such that ${\rm Prm}(R)$ is sober. Then $R$ is an integral domain. 
\end{cor}

\begin{proof}
If $R$ is not an integral domain, then ${\rm Prm}(R)\subsetneq {\rm Prp}(R)$ \cite[Theorem 2.4]{generalized-primary-2}. The conclusion follows again by Proposition \ref{prm-sober-local}.
\end{proof}

\begin{cor}\label{prm-domain2}
	Let $R$ be a one-dimensional Noetherian ring with a unique minimal prime ideal and such that ${\rm Prm}(R)$ is sober. Then $R$ is an integral domain. 
\end{cor}

\begin{proof}
Let $\mathfrak p$ be the unique minimal prime ideal of $R$. Take any maximal ideal $\mathfrak m$ of $R$. By Lemma \ref{prm-localization}, ${\rm Prm}(R_{\mathfrak m})$ is sober and thus $R_{\mathfrak m}$ is an integral domain, by Corollary \ref{prm-domain1}. It follows $\mathfrak{p}R_{\mathfrak m}=0$ and this holds for every maximal ideal $\mathfrak m$ of $R$. Thus $\mathfrak p=0$ and the conclusion follows. 
\end{proof}

\begin{thm}
Let $R$ be a Noetherian ring. Then, the following conditions are equivalent. 
\begin{enumerate}[\upshape (1)]
	\item The space ${\rm Prm}(R)$ is sober. 
	\item The space ${\rm Prm}(R)$ is spectral. 
	\item $R$ is a direct product of zero-dimensional rings and of one-dimensional domains. 
\end{enumerate}
\end{thm}

\begin{proof}
 The equivalence of conditions (1) and (2) immediately follows by Corollary \ref{cor-noet}. 
	
	Suppose that $R=R_1\times \ldots\times R_n$, where each $R_i$ is either a zero-dimensional ring or a one-dimensional domain. Then ${\rm Prm}(R)$ is homeomorphic to the disjoint union of the sober spaces ${\rm Prm}(R_i)$, for $1\leqslant i\leqslant n$ (Proposition \ref{sober-dim1}). Thus ${\rm Prm}(R)$ is sober since it is the disjoint union of finitely many sober spaces. 
	
	Conversely, assume that ${\rm Prm}(R)$ is sober. Then $\dim(R)\leqslant 1$, by Corollary \ref{dim0-1}. In case $\dim(R)=0$, there is nothing to prove. Then we can assume that $\dim(R)=1$. In case $R$ has a unique minimal prime ideal, $R$ is an integral domain, by virtue of Corollary \ref{prm-domain2}, and thus there is nothing to prove. Thus we can assume that $R$ is one-dimensional with $r\geqslant 2$ minimal prime ideals, say $\mathfrak p_1,\ldots, \mathfrak p_r$. We claim that  every maximal ideal of $R$ contains exactly one minimal prime ideal: indeed, if $\mathfrak m$ is a maximal ideal of $R$ and $\mathfrak p_i\neq \mathfrak p_j$ are contained in $\mathfrak m$, then $R_{\mathfrak m}$ would be a one-dimensional Noetherian local ring, and not an integral domain, such that ${\rm Prm}(R_{\mathfrak m})$ is sober (Lemma \ref{prm-localization}), contradicting Corollary \ref{prm-domain1}. The claim immediately implies that the union ${\rm Spec}(R)=\bigcup_{i=1}^r V(\mathfrak p_i)$ is disjoint. In particular, the minimal primes are paiwise comaximal. Let $\mathfrak n$ be the nilradical of $R$ and let $h$ be a positive integer such that $\mathfrak n^h=0$. Then the Chinese Remainder Theorem easily implies that $R$ is isomorphic to $\prod_{i=1}^r R/\mathfrak p_i^h$. Each $R/\mathfrak p_i^h$ has dimension at most $1$. In the latter case, the space ${\rm Prm}(R/\mathfrak p_i^h)$ is homeomorphic to the closed set $\{\mathfrak p_i^h\}^\uparrow\cap{\rm Prm}(R)$ and thus is sober; by Corollary \ref{prm-domain2}, $R/\mathfrak p_i^h$ must be a domain. The conclusion follows. 
\end{proof} 
\textbf{Acknowledgments.} The authors would like to thank the anonymous referee for his/her careful reading and for several remarks that helped to improve the presentation of the paper.
   
\bibliographystyle{plain}
\bibliography{idealspaces}

\end{document}